\documentclass[12pt]{article}
\usepackage{amsmath,amsthm,amssymb,latexsym,color}
\usepackage{bbm}

\usepackage{graphicx}
\usepackage{tikz}
\usetikzlibrary{decorations.pathreplacing,shapes.misc}

\usepackage{pgfplots}

\usepackage{hyperref}

\date{}

\theoremstyle{plain}

\newtheorem*{theorem*}{Theorem}

\newtheorem{theor}{Theorem}[section]

\newtheorem*{theor*}{Theorem}
\newtheorem*{conj*}{Conjecture}

\newtheorem{prop}[theor]{Proposition}

\newtheorem{lemma}[theor]{Lemma}

\theoremstyle{remark}

\def\R{{\mathbb R}}

\def\vr{{\rm vr}}
\def\pvr{{\rm pvr}}
\def\povr{{\rm povr}}
\def\povrs{{\rm povr}^{s}}
\def\absconv{{\rm absconv}}

\newcommand{\lam}{\lambda}

\title{A remark on projection outer volume ratio.}
\author{A. E. Litvak}

\newcommand\address{\noindent\leavevmode

\medskip
\noindent
Alexander  Litvak\\
Dept.~of Math.~and Stat.~Sciences,\\
University of Alberta, \\
Edmonton, AB, Canada, T6G 2G1.\\
\texttt{\small
e-mail:  alitvak@ualberta.ca}
}

\begin{document}

\maketitle

\begin{abstract}
We study the projection outer volume ratio $\povr_k(K)$ introduced in \cite{GLMP}.  
When $k\leq n/2$, we provide sharp bounds in the general case, namely 
\[
       \povr_k(K)\leq
       C \max\Big\{1,  \sqrt{\frac{k}{\ln(2n/k)}}\Big\},
\] where the maximal order is obtained for $K=B_1^n$. 
 Moreover, in the important particular case $K=B_2^n$ we show that 
 $\povr_k(B_2^n)$ is bounded by an absolute constant whenever $k\leq n/2$ and 
\[
   \povr_k(B_2^n)\leq
   C\sqrt{\frac{n}{n-k}}\,
   \ln\left(1+\frac{n}{n-k}\right)
\] for $k>n/2$. The latter bound is sharp up to a logarithmic term. 
\end{abstract}

\bigskip

{\small
}

\section{Introduction}
\label{intro}

The volume ratio is an important affine invariant in the Asymptotic Geometric Analysis
(we refer to \cite{AGM-book} and references therein).  
Given two convex bodies 
$K,L\subset\R^n$ it is defined by
\[
 \vr(K,L)=\inf\left\{
 \left(\frac{|K|}{|T(L)|}\right)^{1/n}:
 T:\R^n\longrightarrow\R^n\ \text{is affine and }T(L)\subset K
 \right\}.
\]
Giannopoulos and Hartzoulaki \cite{GH} proved that for every two convex bodies
$K,L\subset\R^n$,
\[
       \vr(K,L)\leq C\sqrt n\,\ln (n+1).
\]
Recently, Galicer, Merzbacher and Pinasco \cite{GMP26} improved this estimate
to
\begin{equation}\label{vrbest}
       \vr(K,L)\leq C\sqrt{n\ln(n+1)}.
\end{equation}
In the opposite direction, Khrabrov \cite{Kh} proved that for every convex
body $K\subset\R^n$ there exists a convex body $L\subset\R^n$ such that
\[
       \vr(K,L)\geq c\sqrt{\frac{n}{\ln\ln n}}.
\]
This was improved by Galicer, Merzbacher and Pinasco \cite{GMP21}, who showed
that one can choose $L$ centrally symmetric and obtain the sharp lower estimate
\[
       \vr(K,L)\geq c\sqrt n.
\]
The order of this lower bound cannot be improved in general, as by 
John's theorem we have $\vr(B_2^n,L)\leq \sqrt n$ for every convex body $L$.  

It is natural to ask how these estimates change after passing to
lower-dimensional projections.  This question was studied in \cite{GLMP}. 
For $1\leq k\leq n$ and a convex body $K\subset\R^n$, the $k$-projection
volume ratio was defined there by
\[
 \pvr_k(K)=\sup_L\inf_{P,Q}\vr(PK,QL),
\]
where the supremum is taken over all convex bodies $L\subset\R^n$ and the
infimum is taken over all projections $P,Q$ of rank $k$. Here the volume ratio 
 of projections is naturally defined in the following way:
if $E,F$ are $k$-dimensional subspaces of $\R^n$ and
$A\subset E$, $B\subset F$ are convex bodies, then 
\[
 \vr(A,B)=\inf\left\{
 \left(\frac{|A|}{|T(B)|}\right)^{1/k}:
 T:F\longrightarrow E\ \text{is affine and }T(B)\subset A
 \right\}.
\]  Thus $\pvr_k(K)$
measures, in the worst case over $L$, how close suitable $k$-dimensional
projections of $K$ and $L$ can be in terms of volume ratio.  In \cite{GLMP}
quantitative lower bounds for $\pvr_k(K)$ were obtained for every convex body
$K$, together with upper bounds which are sharp up to logarithmic factors in
the relevant ranges.  In particular, the results there show that projection
volume ratios can remain large, more precisely, up to logarithmic factors, 
for every convex body $K$ we have $\pvr_k(K)\geq c \min\{\sqrt{k}, k/\sqrt{n}\}$.

The same paper introduced the natural ``outer'' counterpart of this parameter.
For a convex body $K\subset\R^n$ set
\begin{equation}\label{eq:def-povr}
 \povr_k(K)=\sup_L\inf_{P,Q}\vr(PL,QK),
\end{equation}
where again the supremum is taken over all convex bodies $L\subset\R^n$ and
the infimum is over all projections $P,Q$ of rank $k$.  It was already noted 
in  \cite{GLMP} that one cannot expect a similar lower bound for $\povr_k(K)$ 
uniformly over $K$. Indeed, 
Theorem~6.1 of \cite{GLMP} shows that if $\lam\in (0,1)$ and $k=\lambda n$
then
\[
       \povr_k(B_2^n)\leq C(\lambda).
\]
The proof presented there would lead to a very poor dependence of $C(\lambda)$ on $\lam$, namely 
$C(\lambda)\approx 1/\lam + 1/(1-\lam)$.

The purpose of the present note is to show that for every $k\leq n/2$ one has 
\[
       \povr_k(B_2^n)\leq C,
\]
where $C$ is  an absolute constant and for every $k>n/2$ 
\[
 c\sqrt{\frac{n}{n-k}}
 \leq \povr_k(B_2^n)
 \leq C\sqrt{\frac{n}{n-k}}
 \ln\left(1+\frac{n}{n-k}\right).
\]
Moreover, we show that for every convex body $K$ in $\R^n$ and every $k\leq n/2$,
\[
       \povr_k(K)\leq  
       C \max\Big\{1,  \sqrt{\frac{k}{\ln(2n/k)}}\Big\}.
\]
This bound is sharp, as the case $K=B_1^n$ shows. Note that for $k\geq n/2$ the upper bound 
is of the maximal (up to a logarithm) possible order $\sqrt{n}$ --- by (\ref{vrbest}) 
$$
  \povr_k(K)\leq  C\sqrt{ k \ln (k+1) }.  
$$

Throughout the paper $c,C,c_1,C_1,\ldots$ denote positive absolute constants
whose values may change from line to line, and $B_p^n$ denotes the unit ball
of $\ell_p^n$.

\section{The case $k\leq n/2$}
\label{s:small-k}

We start with the Euclidean ball.

\begin{theor}\label{thm:euclidean-small}
There exists an absolute constant $C>0$ such that for every $n\geq1$ and
$1\leq k\leq n/2$,
\[
   \povr_k(B_2^n)\leq C.
\]
\end{theor}

\begin{proof}
Let $L\subset\R^n$ be a convex body.  Fix an orthogonal projection $P$ of
rank $2k$ and put $E=P\R^n$.  We regard $PL$ as a convex body in the
$2k$-dimensional Euclidean space $E$.

Applying Theorem~6.1 of \cite{GLMP} in the space $E$, with $\lambda=1/2$, we 
find orthogonal projections $S,T:E\to E$ of rank $k$ such that
\[
   \vr(SPL,TB_2^{2k})\leq C.
\]
Since $S$ and $T$ are orthogonal projections onto subspaces of $E$, the
operators $SP$ and $TP$ are orthogonal projections of $\R^n$ of rank $k$.
Moreover, $TB_2^{2k}$ is a Euclidean ball in its range and hence isometric to
$B_2^k$.  Therefore
\[
   \inf_{U,V}\vr(UL,VB_2^n)\leq C,
\]
where the infimum is over all rank $k$ orthogonal projections $U,V$.
Taking the supremum over $L$ proves the result.
\end{proof}

We now pass to an arbitrary convex body. We will use 
both, the Milman's version of the Dvoretzky theorem (we refer to \cite{AGM-book} 
for the history, the non-symmetric case, and references) and 
the isomorphic Dvoretzky theorem (see (\ref{eq:K-eucl}) below), which was proved 
in \cite{MS1,MS2} in the centrally symmetric case, then extended in \cite{GGM} 
and \cite{Gue} to the non-symmetric case for $k\leq  n/\ln n$, and finally 
to the optimal order  in  \cite[Theorem~6.7]{LT}. Note that  
in  \cite{LMT}, the isomorphic Dvoretzky theorem was proved in full generality for random subspaces 
(see Theorem~1 there together with  the remark following its proof). 

\begin{theor}\label{thm:general-small}
There are absolute constants $c,C>0$ such that for every convex body
$K\subset\R^n$,
\[
   \povr_k(K)\leq C, \qquad 1\leq k\leq c\ln n,
\]
and
\[
   \povr_k(K)\leq C\sqrt{\frac{k}{\ln(2n/k)}},
   \qquad c\ln n\leq k\leq n/2.
\]
\end{theor}

\begin{proof}
Fix a convex body $L\subset\R^n$.  By Theorem~\ref{thm:euclidean-small}
there is a rank $k$ orthogonal projection $P$ such that, after identifying its
range with $\R^k$,
\begin{equation}\label{eq:L-eucl}
   \vr(PL,B_2^k)\leq C.
\end{equation}

If $1\leq k\leq c\ln n$, Dvoretzky's theorem (see e.g. \cite{AGM-book}) 
applied to $K$ gives a rank $k$ projection $Q$ which is at an absolute distance from $B_2^k$.
Together with \eqref{eq:L-eucl}, and the multiplicative inequality for the
volume ratio (see e.g. \cite[Fact~2.1]{GLMP}), this gives
\[
   \vr(PL,QK)\leq C,
\]
which proves  \(\povr_k(K)\leq C\).

Assume now that $c\ln n\leq k\leq n/2$.  The isomorphic Dvoretzky theorem
for projections gives a rank $k$ orthogonal projection $Q$ and an affine
image of $QK$, which we again denote by $QK$, such that for some $a, b>0$, 
\begin{equation}\label{eq:K-eucl}
   \frac{1}{b} B_2^k\subset QK\subset a B_2^k,
   \qquad
   ab \leq C\sqrt{\frac{k}{\ln(2n/k)}}.
\end{equation}
Therefore
\[
   \vr(B_2^k,QK)\leq C\sqrt{\frac{k}{\ln(2n/k)}}.
\]
Using again the multiplicative inequality for the volume ratio and
\eqref{eq:L-eucl},
\[
\begin{split}
   \vr(PL,QK)
   &\leq \vr(PL,B_2^k)\,\vr(B_2^k,QK) \leq C\sqrt{\frac{k}{\ln(2n/k)}}.
\end{split}
\]
Taking the infimum over $P,Q$ and then the supremum over $L$ proves the
theorem.
\end{proof}

Next example shows sharpness of Theorem~\ref{thm:general-small}.

\begin{prop}\label{prop:B1-sharp}
There are absolute constants $c,C,c_0>0$ such that
\[
 c\sqrt{\frac{k}{\ln(2n/k)}}
 \leq \povr_k(B_1^n)
 \leq C\sqrt{\frac{k}{\ln(2n/k)}},
 \qquad c_0\ln n\leq k\leq n/2.
\]
\end{prop}

\begin{proof}
The upper estimate follows from Theorem~\ref{thm:general-small}.  For the
lower estimate, choose $L=B_2^n$ in \eqref{eq:def-povr}.  Thus
\[
 \povr_k(B_1^n)
 \geq \inf_{\operatorname{rank}Q=k}\vr(B_2^k,QB_1^n).
\]
Fix such a projection $Q$.  Let $T$ be an affine map satisfying
$T(QB_1^n)\subset B_2^k$.  Since both bodies are symmetric, we can assume that $T$
is linear.
Set $w_i=TQe_i$, $1\leq i\leq n$.  Then $\|w_i\|_2\leq1$ and
\[
   T(QB_1^n)=\absconv\{w_1,\ldots,w_n\}.
\]
The standard volume bound (see \cite{BF,CP,Gl}) yields 
\[
   |T(QB_1^n)|^{1/k}
   \leq C\frac{\sqrt{\ln(2n/k)}}{k}.
\]
Since $|B_2^k|^{1/k}\asymp k^{-1/2}$,
\[
 \left(\frac{|B_2^k|}{|T(QB_1^n)|}\right)^{1/k}
 \geq c\sqrt{\frac{k}{\ln(2n/k)}}.
\]
This proves the lower bound.
\end{proof}

\section{The case $k>n/2$: the Euclidean ball}
\label{s:large-k}

In this section we provide the upper bound on $\povr_k(B_2^n)$ in the regime $k>n/2$, 
which is sharp up to a logarithmic factor. We first note that whenever $K$ is symmetric, 
it is enough to consider only  symmetric bodies $L$ in the definition of $\povr_k(K)$. 
Denote, 
\[
 \povrs_k(K)=\sup_{L=-L}\inf_{P,Q}\vr(PL,QK).
\]

\begin{lemma}\label{lem:sym-reduction}
Let $K\subset\R^n$ be a symmetric convex body and let $1\leq k<n$.
Then
\[
 \povrs_k(K)\leq\povr_k(K)
 \leq \left(2^n{n\choose k}\right)^{1/k}\povrs_k(K).
\]
In particular, if $k>n/2$, then
\[
   \povr_k(K)\leq16\,\povrs_k(K).
\]
\end{lemma}

\begin{proof}
The first inequality and the ``in particular part'' are trivial. 

 Let $L\subset\R^n$ be an arbitrary
convex body.  Translating $L$ if necessary, we may assume that its centroid
is the origin.  Set
\[
   L_0=L\cap(-L).
\]
By the Milman--Pajor inequality \cite[Corollary 3]{MP},
\begin{equation}\label{eq:MP}
   |L_0|\geq2^{-n}|L|.
\end{equation}

Let $P$ be an orthogonal projection of rank $k$ and let $E=P\R^n$. By the 
Fubini theorem,
\[
 |L_0|\leq |PL_0|\,
 \max_{x\in E}|L_0\cap(x+E^\perp)|.
\]
Since $L_0\subset L$, the Rogers--Shephard  inequality
\cite{RS} yields
\[
 |L_0|
 \leq |PL_0|{n\choose k}\frac{|L|}{|PL|}.
\]
Together with \eqref{eq:MP}, this implies
\begin{equation}\label{eq:projection-sym}
 \left(\frac{|PL|}{|PL_0|}\right)^{1/k}
 \leq \left(2^n{n\choose k}\right)^{1/k}.
\end{equation}

Fix rank $k$ projections $P,Q$.  Since $PL_0$ and $QK$ are symmetric, the
affine map in the definition of $\vr(PL_0,QK)$ may be replaced by its linear
part.  Consequently the same image of $QK$ which is contained in $PL_0$ is
also contained in $PL$.  By \eqref{eq:projection-sym},
\[
 \vr(PL,QK)
 \leq \left(2^n{n\choose k}\right)^{1/k}\vr(PL_0,QK).
\]
Taking the infimum over $P,Q$ and then the supremum over $L$ proves the
second inequality. 
\end{proof}

We will also need the following result of Alonso-Guti\'errez, Bastero,
Bernu\'es and Paouris \cite[Theorem~3.2]{ABBP}.

\begin{theor}[Alonso-Guti\'errez--Bastero--Bernu\'es--Paouris]
\label{thm:ABBP}
For every symmetric convex body $L\subset\R^n$ and every
$1\leq k\leq n-1$, there exists a rank $k$ orthogonal projection $P$ such
that
\begin{equation}\label{eq:ABBP}
 \vr(PL,B_2^k)
 \leq C\sqrt{\frac{n}{n-k}}
 \ln\left(1+\frac{n}{n-k}\right),
\end{equation}
where $C>0$ is an absolute constant.
\end{theor}

\begin{theor}\label{thm:euclidean-large}
There exist absolute constants $c,C>0$ such that for every $n/2<k<n$,
\[
 c\sqrt{\frac{n}{n-k}}
 \leq \povr_k(B_2^n)
 \leq C\sqrt{\frac{n}{n-k}}
 \ln\left(1+\frac{n}{n-k}\right).
\]
\end{theor}

\begin{proof}
We first prove the upper estimate.  By Lemma~\ref{lem:sym-reduction}, it is enough to 
consider only symmetric bodies $L$ under the supremum.
For every such $L$, Theorem~\ref{thm:ABBP} provides a rank $k$
projection $P$ with
\[
 \vr(PL,B_2^k)
 \leq C\sqrt{\frac{n}{n-k}}
 \ln\left(1+\frac{n}{n-k}\right).
\]
Every rank $k$ orthogonal projection of $B_2^n$ is a Euclidean ball in its
range, so the upper estimate follows directly from the definition of
$\povr_k(B_2^n)$.

For the lower estimate, take $L=B_\infty^n$ in the definition of
$\povr_k(B_2^n)$.  Since every rank $k$ projection of $B_2^n$ is a Euclidean
ball in its range,
\[
 \povr_k(B_2^n)
 \geq \inf_{\operatorname{rank}P=k}\vr(PB_\infty^n,B_2^k).
\]
The body $PB_\infty^n$ is the unit ball of a $k$-dimensional quotient of
$\ell_\infty^n$.  The classical small-codimension estimates for quotients of
$\ell_\infty^n$ imply
\[
 \inf_{\operatorname{rank}P=k}\vr(PB_\infty^n,B_2^k)
 \geq c\sqrt{\frac{n}{n-k}};
\]
see \cite{GTJ,GMP,GJ} and the references therein.  This proves the lower
estimate and completes the proof.
\end{proof}

\subsection*{Acknowledgments}
The (natural) idea of passing to a $2k$-dimensional subsection in the first step of the proof of Theorem~\ref{thm:euclidean-small} was suggested by ChatGPT.
It also suggested the idea of using Isomorphic Dvoretzky theorem in the general case, however this idea was known to the author.  
Additionally ChatGPT was used for finding references and polishing the presentation, in particular it brought to our attention the reference  
\cite{ABBP} (our original proof of Theorem~\ref{thm:ABBP} was essentially the same as the one in \cite{ABBP}).

\address

\end{document}